\documentclass[letterpaper, 10 pt, conference]{ieeeconf}  

\IEEEoverridecommandlockouts                              
\usepackage{graphicx}
\usepackage{bm}
\usepackage{color}
\newtheorem{theorem}{Theorem}
\newtheorem{lemma}{Lemma}

\newtheorem{remark}{Remark}
\newtheorem{corollary}{Corollary}

\usepackage{amsmath} 
\usepackage{amssymb}  
\usepackage{algorithm,algpseudocode}
\usepackage{graphicx}
\usepackage{booktabs}
\usepackage{longtable}
\allowdisplaybreaks 
\usepackage{tikz}
\usetikzlibrary{graphs}
\ifnum\pdfshellescape=1
\fi
\usetikzlibrary{positioning}

\title{Finite-sample-based Spectral Radius Estimation and Stabilizability Test for Networked Control Systems}

\author{Liang Xu, Baiwei Guo and Giancarlo Ferrari-Trecate
  \thanks{This work was supported by the National Centre of Competence in Research (NCCR) in Dependable and Ubiquitous Automation (grant agreement 51NF40 180545).}
  \thanks{Authors are with the Institute of Mechanical Engineering (IGM),
  EPFL, Switzerland. Email: {\tt\small  \{liang.xu, baiwei.guo, giancarlo.ferraritrecate\}@epfl.ch}}}

\begin{document}

\maketitle
\thispagestyle{empty}
\pagestyle{empty}

\begin{abstract}
  In the analysis and control of discrete-time linear time-invariant systems, the spectral radius of the system state matrix plays an essential role.
Usually, it is assumed that system matrices are known, from which the spectral radius can be directly computed.
Instead, we consider the setting where the system is affected by process noise, and one has only finitely many samples of system input and state measurements.
We provide two methods for estimating the spectral radius and derive error bounds that hold with high probability.
Moreover, we show how to use the derived results to test stabilizability for networked control systems (NCSs) with lossy channels when only finitely many samples of the system input, state, and packet drop sequence are available.
\end{abstract}

\section{Introduction}
Due to the development of sensing, communication, and computation techniques, data availability is steadily increasing, which has motivated a surge of research in directly using data for control system analysis and design.
For example, without explicit system identification, finite-sample-based criteria are proposed to evaluate the controllability~\cite{RN11260, RN10935}, passivity/dissipativity~\cite{RN11100,RN11098,RN11099,RN11095} and $L_2$ gain~\cite{RN11096} of the underlying dynamical systems.

This paper considers using data to estimate the spectral radius of the system state matrix and check the stabilizability of networked control systems (NCSs).
The knowledge of the spectral radius of the system state matrix is fundamental in linear system analysis and control design.
For example, in networked control over channels affected by packet drops, the spectral radius of the open-loop system state matrix must satisfy specific conditions depending on the packet reception rate to ensure stabilizability~\cite{RN373, RN2897}.
In various applications, the system parameters are reconstructed from a finite amount of input-output or input-state data.
Estimates of the spectral radius derived from noisy data are subject to errors and thus may lead to wrong conclusions about critical properties of the system.
For example, in networked control over channels affected by packet drops, if the stabilizability margin is small, a significant error in the spectral radius may result in an erroneous conclusion about stabilizability.

The work~\cite{RN11554} and~\cite{RN11389} show how to determine stability and stabilizability properties of linear systems from a finite amount of data.
The authors of~\cite{RN11554} aim to assess the stability of an unknown deterministic switched linear system from a finite number of trajectories.
They provide a chance-constrained optimization-based approach and give probabilistic stability guarantees.
For networked control over lossy channels,~\cite{RN11389} shows how to use finitely many samples of the packet drop sequence to estimate the channel statistics and further determine whether stabilizability conditions for NCSs hold.
In this paper, we address the same problem as in~\cite{RN11389}  but in the more challenging scenario where both the channel statistics and the system matrices are unknown.

We estimate the spectral radius from finitely many input-state data, propose two estimation algorithms, and derive small error bounds which hold with high probability by leveraging recent results on non-asymptotic least-square system identification~\cite{RN10374}~\cite{RN10762} and matrix eigenvalue perturbation techniques~\cite{RN11408}.
The derived sample complexity bound scales linearly with the system dimension.
Moreover, based on the above methods, we propose an algorithm to determine whether the stabilizability condition for networked control over lossy channels holds from a finite amount of system input, state, and packet drop samples.
The proposed algorithm uses the estimated spectral radius and the packet reception rate to check whether a suitably defined stabilizability inequality holds.
We show that the sample complexity for testing stabilizability scales in a way inversely proportional to the square of the stabilizability margin.

This paper is organized as follows.
In Section~\ref{sec.ProblemFormulation}, we give the problem formulation.
The finite-sample estimation of the spectral radius is studied in Section~\ref{sec.SepctralRadiusEstimation}.
The data-based NCS stabilizability test is provided in Section~\ref{sec.StabilityTest}.
Simulations are given in Section~\ref{sec.Simulation} and conclusions are provided in Section~\ref{sec.Conclusion}.

\textbf{Notation}:
$\rho(\cdot)$ denotes the spectral radius of a matrix.
$\lambda_{\max}(\cdot) (\lambda_{\min}(\cdot))$ denotes the maximal (minimal) eigenvalue of a positive semi-definite matrix.
$\|A\|$ represents the spectral norm of the matrix $A$.
$\log(x)$ denotes the logarithm of $x$ in base $e$.

\section{Problem Formulation\label{sec.ProblemFormulation}}

Consider the LTI system
\begin{align}
  \label{eq.GaussianProcess}
  x_{t+1}=Ax_t+Bu_t+w_t,
\end{align}
where $x_t \in \mathbb{R}^{n}$, $u_t\in \mathbb{R}^{p}$, $w_t \sim \mathcal{N}(0, \sigma_w^2 I)$ and the process noises $\{w_0, w_1, \ldots\}$ are independent.
We assume that matrices $A$ and $B$ are unknown and address the problem of estimating the spectral radius $\rho(A)$  from collected input and state measurements and of providing probabilistic bounds for the estimation error.

Moreover, we consider the networked control setting where sensors measure the system state $x_t$ and transmit it to the controller through a wireless communication channel,  suffering from Bernoulli packet drops~\cite{RN373} with packet reception rate $q\in (0,1)$.
The packet drop event is denoted by a binary variable $\gamma_t$, which is $0$ if the packet is lost and $1$ otherwise.
As shown in~\cite{RN373}, there exist an estimator and a controller such that closed-loop control system is mean square stable, only if the following condition holds
\begin{align}
  \label{eq.StabilityCondition}
  q> 1-\frac{1}{\rho(A)^2}.
\end{align}
We assume that the precise value of  $q$ and  $\rho(A)$ are unknown.
We aim to provide a method of using the collected input/state measurements and communication channel samples to determine whether the stabilizability condition~\eqref{eq.StabilityCondition} holds.
In the presence of a finite amount of data, due to estimation errors, we can never correctly assert whether the equality    $q=1-1/{\rho(A)^2}$ holds.
Therefore, throughout this paper,  we assume that  $q \neq 1-1/{\rho(A)^2}$.

\section{Finite-Sample-Based Spectral Radius Estimation\label{sec.SepctralRadiusEstimation}}
This section provides two algorithms for estimating the spectral radius from a finite amount of input-state samples of~\eqref{eq.GaussianProcess} and derive estimation error bounds.
The first algorithm utilizes all data samples to generate the estimate, and the associated performance bound is data-dependent.
The second method requires multiple system trajectories and uses the last input and the last two state measurements in each trajectory to generate the estimate, which allows deriving a data-independent estimation error bound.

The first algorithm is more data-efficient than the second one, and the associated bounds can be tighter.
However, the sample complexity analysis is prohibitive. 
In contrast, the error bound of the second algorithm is a function of system parameters and the number of trajectories $N$ only.
The explicit dependence on $N$ allows deriving sample complexity results, i.e., to determine how many system trajectories are needed to achieve the desired estimation accuracy.

\subsection{Data-Dependent Estimation Error Bounds}
We consider the estimation procedure in Algorithm~\ref{alg.DataDependentSpectralRadiusEstimation}.
\begin{algorithm}[htbp]
  \caption{Finite-sample-based Spectral Radius Estimation}
  \label{alg.DataDependentSpectralRadiusEstimation}
  \begin{algorithmic}[1]
    \Require collect independent data $\{u^{(l)}, x^{(l)}, x_+^{(l)}\}$ for $l=1, \ldots, M$, where $x_+^{(l)}=Ax^{(l)}+Bu^{(l)}+w^{(l)}$
\State use least-squares to estimate the system state matrix $(A, B)$, i.e.,
\begin{align}\label{eq.DataDependentLeastSquare}
  & ( \hat{A}, \hat{B})=\arg \min_{A, B} \sum_{l=1}^M
     \| x_+^{(l)}-Ax^{(l)}-Bu^{(l)}\|^2\\
  &=\left(\sum_{l=1}^M x_+^{(l)}
    \begin{bmatrix}
      x^{(l)}\\
          u^{(l)}
    \end{bmatrix}^\top\right)  \left(\sum_{l=1}^M \begin{bmatrix}
      x^{(l)}\\
          u^{(l)}
    \end{bmatrix}\begin{bmatrix}
      x^{(l)}\\
          u^{(l)}
    \end{bmatrix}^\top\right)^{-1} \nonumber
\end{align}
\State  calculate $\rho(\hat{A})$ and return $\rho(\hat{A})$ as an estimate of $\rho(A)$.
  \end{algorithmic}
\end{algorithm}

  In Algorithm~\ref{alg.DataDependentSpectralRadiusEstimation}, independent data tuple $\{u^{(l)}, x^{(l)}, x_+^{(l)}\}$ means the random samples $w^{(l)}=x_+^{(l)}-Ax^{(l)}-Bu^{(l)}$  are independent.
  Independent data tuples can be constructed from a trajectory of~\eqref{eq.GaussianProcess} by using random inputs since the process samples at different times are independent.
  Besides, even though in~\eqref{eq.DataDependentLeastSquare} we also obtain an estimate $\hat{B}$ for the input matrix, it will not be used in this paper.


The next lemma provides data-dependent error bounds for the least-square estimate~\eqref{eq.DataDependentLeastSquare} in Algorithm~\ref{alg.DataDependentSpectralRadiusEstimation}.
\begin{lemma}[Proposition 3 of~\cite{RN10374}]\label{lem:DataDependentLST}
  Assume that $M \geq n+p$, then, with probability at least $1-\delta$,   the least-square estimate in~\eqref{eq.DataDependentLeastSquare} verifies
\begin{multline*}
\left[\begin{array}{l}
(\widehat{A}-A)^{\top} \\
(\widehat{B}-B)^{\top}
      \end{array}\right]
    \begin{bmatrix}
   (\widehat{A}-A) & (\widehat{B}-B)
      \end{bmatrix}
    \preceq C(n, p, \delta)\Phi^{-1}
 \end{multline*}
 where $\Phi=
\sum_{l=1}^M\left[ \begin{smallmatrix}
      x^{(l)}\\
          u^{(l)}
    \end{smallmatrix}\right] \left[\begin{smallmatrix}
      x^{(l)}\\
          u^{(l)}
        \end{smallmatrix}\right]^\top $, $C(n, p, \delta)=\sigma_{w}^{2}(\sqrt{n+p}+\sqrt{n}+\sqrt{2 \log (1 / \delta)})^{2}$.
      If $\Phi$ has a zero eigenvalue, we set $C(n, p, \delta)\Phi^{-1}=+\infty$. 
\end{lemma}
From Lemma~\ref{lem:DataDependentLST}, we can obtain an error bound for estimating $A$. Define $E=[I, 0]$. Then, we have
$
(\widehat{A}-A)^{\top} (\widehat{A}-A)
    \preceq C(n, p, \delta)E \Phi^{-1} E^{\top}
 $,
which gives
\begin{align}
\|\widehat{A}-A\| 
  \le C(n, p, \delta)^{\frac{1}{2}} \lambda_{\max}^{\frac{1}{2}} \left( E \Phi^{-1} E^{\top}\right)
   := f_1(\delta).\label{eq:f1Def}
 \end{align}

The spectral radius estimation error bound of Algorithm~\ref{alg.DataDependentSpectralRadiusEstimation} will be based on~\eqref{eq:f1Def} and the following eigenvalue perturbation result.
 \begin{lemma}[Theorem VIII.1.1 in~\cite{RN11408}]\label{lem.SpectralVariationBounds}
Suppose $P$ and $Q$ are two $n \times n$ matrices with eigenvalues $\alpha_{1}, \ldots, \alpha_{n},$ and $\beta_{1}, \ldots, \beta_{n},$ respectively, then we have
\begin{align}
  \label{eq:1}
\max _{j} \min _{i}\left|\alpha_{i}-\beta_{j}\right|  \leq  (\|P\|+\|Q\|)^{1-\frac{1}{n}}\|P-Q\|^{\frac{1}{n}}.
\end{align}
Moreover, when $P=Q$ or  $P=-Q=I$, the equality in~\eqref{eq:1} holds.
\end{lemma}

\begin{theorem}\label{thm.SpectralDependentBounds}
  Consider Algorithm~\ref{alg.DataDependentSpectralRadiusEstimation}, fix a failure probability $\delta \in(0,1),$ and assume that
 $M \geq n+p$.
  Then, it holds with probability at least $1-\delta$ that
\begin{align}
  |\rho(A)-\rho(\hat{A})|& \le    {\left(2\|\hat{A}\|+ f_1(\delta)\right)}^{1-\frac{1}{n}}{ \left(f_1(\delta)\right)}^{\frac{1}{n}}\nonumber\\
 & := f_2(\delta), \label{eq.f4Def}
\end{align}
where $f_1(\delta)$ is defined in~\eqref{eq:f1Def}.
\end{theorem}

\begin{proof}
Suppose the eigenvalues of $A$ and $\hat{A}$ are $\alpha_{1}, \ldots, \alpha_{n},$ and $\beta_{1}, \ldots, \beta_{n}$, respectively.
Without loss of generality, assume $\beta_{j}$ are arranged in magnitude decreasing order, i.e.
$  \rho(\hat{A})=|\beta_1|\ge |\beta_2|\ge \ldots \ge |\beta_n|.$
Then from Lemma~\ref{lem.SpectralVariationBounds},  for all $j$, we have
\begin{align*}
  |\alpha_{\sigma(j)}-\beta_{j}|    \leq   (\|A\|+\|\hat{A}\|)^{1-\frac{1}{n}}\|A-\hat{A}\|^{\frac{1}{n}},
\end{align*}
where $\sigma(j)=\arg \min_i |\alpha_i-\beta_j|$.
In view of the above inequality we have
\begin{align*}
|\beta_j| \le   (\|A\|+\|\hat{A}\|)^{1-\frac{1}{n}}\|A-\hat{A}\|^{\frac{1}{n}}+|\alpha_{\sigma(j)}|.
\end{align*}
Let $j=1$, we have
\begin{align*}
  \rho(\hat{A})&=|\beta_1|\le   (\|A\|+\|\hat{A}\|)^{1-\frac{1}{n}}\|A-\hat{A}\|^{\frac{1}{n}} +|\alpha_{\sigma(1)}| \\
  &\le   (\|A\|+\|\hat{A}\|)^{1-\frac{1}{n}}\|A-\hat{A}\|^{\frac{1}{n}}+\rho(A).
\end{align*}
Since in Lemma~\ref{lem.SpectralVariationBounds}, the matrices $P$ and $Q$ can be exchanged, following similar derivations, we can also obtain that
\begin{align*}
  \rho(A)\le   (\|A\|+\|\hat{A}\|)^{1-\frac{1}{n}}\|A-\hat{A}\|^{\frac{1}{n}}+\rho(\hat{A}).
\end{align*}
Therefore, we conclude that
\begin{align}\label{eq.ImmediateErrorBounds}
  |\rho(A)-\rho(\hat{A})| \le   (\|A\|+\|\hat{A}\|)^{1-\frac{1}{n}}\|A-\hat{A}\|^{\frac{1}{n}}.
\end{align}

Since
\begin{align}\label{eq.hatAscaling}
  \|A\|=\|\hat{A}+A-\hat{A}\|\le \|\hat{A}\|+\|A-\hat{A}\|,
\end{align}
we have
\begin{align*}
  |\rho(A)-\rho(\hat{A})| \le    (2\|\hat{A}\|+\|A-\hat{A}\|)^{1-\frac{1}{n}} \|A-\hat{A}\|^{\frac{1}{n}}.
\end{align*}
Further, from the bound on the error norm in~\eqref{eq:f1Def} and Lemma~\ref{lem:DataDependentLST}, we can obtain the result.
\end{proof}

Theorem~\ref{thm.SpectralDependentBounds} is a non-trivial extension of existing finite-sample-based system identification results~\cite{RN10374} and shows how system parameters affect the estimation performance.

\subsection{Data-Independent Estimation Error Bounds}\label{sec.DataIndependentSpectralEstimate}
Suppose we collect the following samples of multiple trajectories from system~\eqref{eq.GaussianProcess} initialized at $x_0=0$ using i.i.d.\ inputs $u_t \sim \mathcal{N}(0, \sigma_u^2I)$,
\begin{align*}
  \{u_0^{(i)},x_0^{(i)}, u_1^{(i)},x_1^{(i)}, \ldots, u_{T}^{(i)}, x_T^{(i)}, x_{T+1}^{(i)}\},
\end{align*}
where $i=1, \ldots, N$ denotes the index of $N$ independent experiments.
To estimate the spectral radius, we propose Algorithm~\ref{alg.DataIndependentSpectralRadiusEstimation}, which is based on the least-square identification algorithms analyzed in~\cite{RN10374}.

\begin{algorithm}[htbp]
  \caption{Finite-sample-based Spectral Radius Estimation Using Multiple Trajectories}
  \label{alg.DataIndependentSpectralRadiusEstimation}
  \begin{algorithmic}[1]
    \Require collected data $\{u_T^{(i)}, x_T^{(i)}, x_{T+1}^{(i)}\}$ for $i=1, \ldots, N$
\State use least-squares to estimate the system state matrix $(A, B)$, i.e.,
\begin{align}\label{eq.MultiTrajLeastSquare}
& ( \hat{A}, \hat{B})=\arg \min_{A, B} \sum_{i=1}^N\| x_{T+1}^{(i)}-Ax_T^{(i)}-Bu_T^{(i)}\|^2\\
  &=\left(\sum_{i=1}^N x_{T+1}^{(i)}
    \begin{bmatrix}
      x_T^{(i)}\\
          u_T^{(i)}
    \end{bmatrix}^\top\right)  \left(\sum_{i=1}^N \begin{bmatrix}
      x_T^{(i)}\\
          u_T^{(i)}
    \end{bmatrix}\begin{bmatrix}
      x_T^{(i)}\\
          u_T^{(i)}
    \end{bmatrix}^\top\right)^{-1} \nonumber
\end{align}
\State  calculate $\rho(\hat{A})$ and return $\rho(\hat{A})$ as an estimate of $\rho(A)$.
  \end{algorithmic}
\end{algorithm}

The assumption that the system is initialized at $x_0=0$ is made to simplify the finite-sample analysis of the least-square estimation~\eqref{eq.MultiTrajLeastSquare}, see~\cite{RN10374}.
This assumption can be relaxed, and we can follow similar proof steps as in~\cite{RN10374} to conduct finite-sample analysis  of~\eqref{eq.MultiTrajLeastSquare} for data collected from $x_0\neq 0$.
One can also use all measurements from each trajectory to obtain an estimate of $\rho(A)$.
  However, the sample complexity analysis of least squares for this approach relies on the knowledge of the eigenvalue distribution of the state matrix~\cite{RN10762, RN10374, RN11035, RN11104}, and therefore, cannot be applied when this prior information is not available.
  Moreover, as shown in~\cite{RN11104},  least-square algorithms using all measurements from a single trajectory are inconsistent for specific system matrices.
  Even though Algorithm~\ref{alg.DataIndependentSpectralRadiusEstimation} only uses the last input and the last two state measurements in each trajectory, it can still provide a tight estimate, compared with Algorithm~\ref{alg.DataDependentSpectralRadiusEstimation}, which can be observed from simulations in Section~\ref{sec.Simulation}.  
  This is because the state measurements in a trajectory are statistically correlated, and the last two state measurements also contain information about previous state measurements.
  A corroborating fact is that the error bound~\eqref{eq:f3Df} (below in Lemma~\ref{lem.LeastSquareIdentification}) for estimating system matrices using last input and last two state measurements decreases as the total trajectory length $T$ increases, which means that a longer trajectory will result in a smaller estimation error even if not all the measurements are used.

  In the sequel, we analyze the performance of Algorithm~\ref{alg.DataIndependentSpectralRadiusEstimation}.
  We first introduce a preliminary result on the sample complexity of~\eqref{eq.MultiTrajLeastSquare}.
\begin{lemma}[Theorem III.3 in~\cite{RN10762}]\label{lem.LeastSquareIdentification}
  Consider the least-squares estimator~\eqref{eq.MultiTrajLeastSquare}, 
  fix a failure probability $\delta \in(0,1),$ and assume that $N \geq 8(n+p)+16\log (4 / \delta)$.
  Then, it holds with probability at least $1-\delta$ that
  \begin{align}
    \|\hat{A}-A\|& \leq 16 \sigma_w \lambda_{\min }^{-1 / 2}\left(\Sigma\right) \sqrt{\frac{ (n+2p) \log (36 / \delta)}{N}} \nonumber \\ 
                     &:= f_3(N, \delta), \label{eq:f3Df}
  \end{align}
where $\Sigma:= \sum_{t=0}^{T} \sigma_u^2A^{t}BB^\top \left(A^{\top}\right)^{t}+ \sigma_w^2 A^{t}\left(A^{\top}\right)^{t}.$
\end{lemma}

The next theorem use Lemma~\ref{lem.LeastSquareIdentification} for providing a spectral radius estimation error bound for Algorithm~\ref{alg.DataIndependentSpectralRadiusEstimation}.
The proof is similar to that of Theorem~\ref{thm.SpectralDependentBounds} and is omitted here for brevity.
\begin{theorem}\label{thm.SpectralIndependentBounds}
  Consider Algorithm~\ref{alg.DataIndependentSpectralRadiusEstimation}, fix a failure probability $\delta \in(0,1),$ and assume that
 $N \geq 8(n+p)+16\log (4 / \delta)$.
  Then, it holds with probability at least $1-\delta,$ that
\begin{align}
  |\rho(A)-\rho(\hat{A})|& \le    {\left(2\|A\|+ f_3(N, \delta)\right)}^{1-\frac{1}{n}}{ \left(f_3(N,\delta)\right)}^{\frac{1}{n}}\nonumber \\
  &:= f_4(N, \delta), \label{eq.f4Def}
\end{align}
where $f_3(N,\delta)$ is defined in Lemma~\ref{lem.LeastSquareIdentification}.
\end{theorem}

  Theorem~\ref{thm.SpectralIndependentBounds} explicitly shows how various system parameters affect the spectral radius estimation error.
Even though the error bound~\eqref{eq.f4Def} involves $\|A\|$ and $\Sigma$,  in practice, we only need to know an upper bound of $\|A\|$ and a lower bound of $\Sigma$ to use~\eqref{eq.f4Def}.
From Theorem~\ref{thm.SpectralIndependentBounds}, we can also derive the sample complexity bound for estimating the spectral radius using Algorithm~\ref{alg.DataIndependentSpectralRadiusEstimation}.
It is clear from Theorem~\ref{thm.SpectralIndependentBounds} that $f_4$ is an increasing function of $f_3$ and, if $N$ is sufficiently large, $f_3$ and $f_4$ can be made arbitrarily small, which means any desired estimation precision can be achieved if the number of trajectories $N$ is large enough.
However,  explicitly calculating the sample complexity bound of $N$ that achieves a given accuracy $\epsilon$ for general systems from~\eqref{eq.f4Def} is a difficult task.
In the following, we consider the case that  $\epsilon-2(1-\frac{1}{n})\|A\| >0$ and provide a sample complexity result for
Algorithm~\ref{alg.DataIndependentSpectralRadiusEstimation}.
\begin{corollary}\label{cor:SpectralRadiusEstimationSampleComplexity}
Consider Algorithm~\ref{alg.DataIndependentSpectralRadiusEstimation}, let  $b=\epsilon-2(1-\frac{1}{n})\|A\| $, for given $\epsilon, \delta$, if $b >0$ and  $N$ is greater than 
$\max \left\{ N_1,  8(n+p)+16\log \frac{4}{\delta}\right\}$,
  where $N_1=\frac{256 \sigma_w^2  (n+2p)}{b^2\lambda_{\min }(\Sigma)}\log\frac{36}{\delta}$, then with probability at least $1-\delta$, we have
  \begin{align}
  |\rho(A)-\rho(\hat{A})|\le \epsilon.\label{eq:2}
  \end{align}
\end{corollary}
\begin{proof}
Since $b>0$ and $N\ge N_1$, we have that $f_3(N,\delta)\le b$.
Substituting the expression of $b$ and taking logarithm of both sides, we obtain
$ \log \left( 2(1-\frac{1}{n}) \|A\| +f_3 \right) \le \log \epsilon$.
From Jensen's inequality, we have
\begin{align*}
  &\log \left(  2(1-\frac{1}{n}) \|A\| +f_3 \right)\\
  &=\log \left(  (1-\frac{1}{n}) (2\|A\| +f_3) +\frac{1}{n} f_3 \right) \\
  & \ge   (1-\frac{1}{n}) \log \left( 2\|A\| +f_3 \right)+ \frac{1}{n} \log f_3.
\end{align*}
Therefore, we have
$  (1-\frac{1}{n}) \log \left( 2\|A\| +f_3 \right)+ \frac{1}{n} \log f_3 \le \log \epsilon$,
which means
$  f_4(N, \delta) = (2\|A\|+f_3)^{1-\frac{1}{n}} f_3^{\frac{1}{n}}  \le \epsilon$.
In view of Theorem~\ref{thm.SpectralIndependentBounds}, we know that~\eqref{eq:2} holds with probability at least $1-\delta$.
\end{proof}

It is clear from Corollary~\ref{cor:SpectralRadiusEstimationSampleComplexity} that, when $n$ is large,  $b\approx \epsilon -2 \|A\|$. Then the sample complexity bound $N_1$ scales linearly with the sum of state and input dimensions $n+2p$.

\section{Finite-Sample-Based Stabilizability Test for NCSs\label{sec.StabilityTest}}

Suppose we already know that with probability at least $1-\delta$ that the following error bound holds
\begin{align}
|\rho(\hat{A})-  \rho(A)|<\epsilon, \label{eq.temp4}
\end{align}
which can be obtained either from Algorithm~\ref{alg.DataDependentSpectralRadiusEstimation}/Theorem~\ref{thm.SpectralDependentBounds} or Algorithm~\ref{alg.DataIndependentSpectralRadiusEstimation}/Theorem~\ref{thm.SpectralIndependentBounds}.
In this section, we show that if the packet reception rate $q$ is unknown and a finite amount of channel samples are available, one can check from these data whether the stabilizability condition~\eqref{eq.StabilityCondition} holds.
The following lemma, which is a variant of Hoeffding's inequality, provides error bounds for estimating $q$ from finitely many samples.
\begin{lemma}[Theorem 4.5 of~\cite{RN10404}]\label{lem.qFiniteSample}
  Consider a sequence $\{\gamma_k, k=0, \ldots, N_q-1\}$ of i.i.d.\ random variables taking values in $\{0,1\}$ with mean $q$.
  Let $\hat{q}=\frac{1}{N_q}\sum_{k=0}^{N_q-1}\gamma_k$ be the sample average.
  Then for any $\delta_q\in (0,1)$, it holds that
  \begin{align*}
    \Pr \left(|q-\hat{q}|\le f_5(N_q, \delta_q) \right)\ge 1-\delta_q,
  \end{align*}
  where $f_5(N_q, \delta_q):=\sqrt{\frac{ \log (2/\delta_q)}{2N_q}}$.
\end{lemma}

Based on the above result, Algorithm~\ref{alg.NCSStabilityTest} is proposed to determine the stabilizability of a NCS.
\begin{algorithm}[htbp]
  \caption{Finite-Sample-Based  Stabilizability Test for NCSs}
  \label{alg.NCSStabilityTest}
  \begin{algorithmic}[1]
\Require  $\delta_q$, $\delta$, $\epsilon$, $\rho(\hat{A})$, collected data $\{\gamma_k\}$ for $k=0, \ldots, N_q-1$
\State calculate $\hat{q}=\frac{1}{N}\sum_{k=0}^{N_q-1}\gamma_k$ and $f_5(N_q, \delta_q)$
\If {$\hat{q}-f_5(N_q, \delta_q)<1$ and $\rho(\hat{A}) +\epsilon< 1/\sqrt{{1-\hat{q}+f_5(N_q, \delta_q)}}$}
\State return ``stabilizability condition~\eqref{eq.StabilityCondition} holds''
\ElsIf  {$\hat{q}+f_5(N_q, \delta_q)<1$ and $\rho(\hat{A}) -\epsilon> 1/\sqrt{{1-\hat{q}-f_5(N_q, \delta_q)}}$}
\State return ``stabilizability condition~\eqref{eq.StabilityCondition} does not hold''
\Else
\State return  ``undetermined''
\EndIf
\end{algorithmic}
\end{algorithm}

The performance analysis of Algorithm~\ref{alg.NCSStabilityTest} is given in the next theorem.
\begin{theorem}\label{thm.StabilityTestAnalysis}
If $N_q, \delta_q$ and $\epsilon$ satisfy
  \begin{gather}
    f_5(N_q, \delta_q)<\frac12(1-q),    \label{eq:NqStar1}\\
        f_5(N_q, \delta_q) < \frac12 |1-q-\frac{1}{\rho(A)^2}|,  \label{eq:NqStar2}\\
\epsilon\le
  \frac{1}{2}   \max\left\{  1/\sqrt{{1-q+2f_5(N_q, \delta_q)}}-\rho(A)\right.,\nonumber \\\quad \quad \quad\quad \quad \quad\quad \left. \rho(A) -1/\sqrt{{1-q-2f_5(N_q, \delta_q)}} \right\},   \label{eq.Nstar2} 
\end{gather}
 then Algorithm~\ref{alg.NCSStabilityTest} returns the correct answer\footnote{``returns the correct answer'' means: if \eqref{eq.StabilityCondition} holds, Algorithm~\ref{alg.NCSStabilityTest} returns ``stabilizability condition~\eqref{eq.StabilityCondition} holds"; if \eqref{eq.StabilityCondition} does not hold, Algorithm~\ref{alg.NCSStabilityTest} returns ``stabilizability condition~\eqref{eq.StabilityCondition} does not hold". } with probability at least
$(1-\delta)(1-\delta_q)$.  
\end{theorem}
\begin{proof}We first prove the result when~\eqref{eq.StabilityCondition} holds, i.e., $\rho(A)<\sqrt{\frac{1}{1-q}}$.
From~\eqref{eq:NqStar1}, we know that
$   1/ \sqrt{{1-q+2f_5(N_q, \delta_q)}}>0$.
Moreover, from~\eqref{eq:NqStar2} and~\eqref{eq.Nstar2}, we have
  \begin{align}
    2 \epsilon \le  1/\sqrt{{1-q+2f_5(N_q, \delta_q)}}-\rho(A) \label{eq.EpsilonBound}.
  \end{align}
The probability that the algorithm returns the correct answer  ``stabilizability condition~\eqref{eq.StabilityCondition} holds'' is
  \begin{align*}
    &\Pr\left(\rho(\hat{A}) +\epsilon< 1/\sqrt{{1-\hat{q}+f_5}}\right)\\
    &\overset{(a)}{\ge}  \Pr\left(1/\sqrt{{1-\hat{q}+f_5}}> 1/\sqrt{{1-q+2f_5}}\right) \\
    & \quad \times \Pr \left(  1/\sqrt{{1-q+2f_5}}>\rho(\hat{A}) +\epsilon\right)\\
    &\ge  \Pr(q<\hat{q}+f_5<q+2f_5)\\
    &\quad \times \Pr \left(  1/\sqrt{{1-q+2f_5}}-\rho(A)>\rho(\hat{A})-\rho(A) +\epsilon \right)\\
    &\overset{(b)}{\ge} (1-\delta_q)\\
    &\quad \times\Pr \left(1/\sqrt{{1-q+2f_5}}-\rho(A)>\rho(\hat{A})-\rho(A) +\epsilon\right)\\
    & = (1-\delta_q) \times \\
    &  \left(1-\Pr\left(\rho(\hat{A})-\rho(A)\ge 1/\sqrt{{1-q+2f_5}}-\rho(A)-\epsilon \right)\right)\\
    & \overset{(c)}{\ge} (1-\delta_q) \left(1-\Pr\left(\rho(\hat{A})-\rho(A)\ge \epsilon \right)\right)\\
& \ge (1-\delta_q)(1-\delta),
  \end{align*}
  where $(a)$ follows from the fact that  for any scalar random variables $y, z$ and a given constant $\alpha$, the following probability relation holds
$
\Pr(z<y)
  =\Pr(z<y<\alpha)+\Pr(z<\alpha<y)+\Pr(\alpha<z<y)
  \ge \Pr(z<\alpha<y)$;
  $(b)$ follows from Lemma~\ref{lem.qFiniteSample} and $(c)$ follows from~\eqref{eq.EpsilonBound}.
  
  Next we prove the result when~\eqref{eq.StabilityCondition} does not hold, i.e., $\rho(A)>1/\sqrt{{1-q}}$.
  The proof is similar to the case that $\rho(A)<1/\sqrt{{1-q}}$ and we only outline the main steps.
The probability that the algorithm returns the correct answer ``stabilizability condition~\eqref{eq.StabilityCondition} does not hold'' is
  \begin{align*}
    &\Pr\left(\rho(\hat{A}) - \epsilon> 1/\sqrt{{1-\hat{q}-f_5}}\right)\\
    &\ge  \Pr\left(  1/\sqrt{{1-q-2f_5}} >1/\sqrt{{1-\hat{q}-f_5}}\right) \\
    & \quad \times \Pr \left(\rho(\hat{A}) -\epsilon>  1/\sqrt{{1-q-2f_5}}\right)\\
    &\ge  \Pr(q+2f_5>\hat{q}+f_5>q)\\
    &\quad \times  \Pr \left(\rho(\hat{A}) -\rho(A) >  1/\sqrt{{1-q-2f_5}}-\rho(A)+\epsilon\right)\\
    &\ge (1-\delta_q)\\
    &\quad \times\Pr \left(\rho(\hat{A}) -\rho(A) >  1/\sqrt{{1-q-2f_5}}-\rho(A)+\epsilon\right)\\
    & \ge (1-\delta_q)\\
    &\quad \times \left(1-\Pr\left(\rho(\hat{A})-\rho(A)\ge 1/\sqrt{{1-q-2f_5}}-\rho(A)+\epsilon \right)\right)\\
    & \ge (1-\delta_q)\left(1-\Pr\left(\rho(\hat{A})-\rho(A)\ge \epsilon \right)\right)
 \ge (1-\delta_q)(1-\delta).
  \end{align*}
The proof is complete.
\end{proof}

Theorem~\ref{thm.StabilityTestAnalysis} characterizes conditions on samples such that Algorithm~\ref{alg.NCSStabilityTest} can return a correct answer with high probability. 
Besides, the result in Theorem~\ref{thm.StabilityTestAnalysis} analytically shows how various system parameters affect the performance of Algorithm~\ref{alg.NCSStabilityTest}.
The right-hand sides of~\eqref{eq:NqStar2} and~\eqref{eq.Nstar2} can be interpreted as stabilizability margins, and imply that if the stabilizability margin is small, one should require small estimation errors for $\rho(A)$ and $q$.
  Moreover, the right-hand side of~\eqref{eq.Nstar2} is a decreasing function of $f_5$. This means that if $q$ is estimated with high accuracy, the requirement for a small $\epsilon$ can be relaxed, which represents a trade-off between the estimation accuracy of $q$ and $\rho(A)$.
  Even though applying Theorem~\ref{thm.StabilityTestAnalysis} requires to know $q$ and $\rho(A)$, in practice, if we have a prior knowledge of the stabilizability margin $|1-q-\frac{1}{\rho(A)^2}|$, we can collect enough samples to make $f_5$ and $\epsilon$ sufficiently small and therefore enforcing~\eqref{eq:NqStar1}, \eqref{eq:NqStar2}, and \eqref{eq.Nstar2}  to  use Theorem~\ref{thm.StabilityTestAnalysis}.

Moreover, if we use Algorithm~\ref{alg.DataIndependentSpectralRadiusEstimation} to estimate the spectral radius, we have the following sample complexity results from Theorem~\ref{thm.StabilityTestAnalysis} and Corollary~\ref{cor:SpectralRadiusEstimationSampleComplexity}.
\begin{corollary}\label{cor.StabilizabilityTestSampleComplexity}
Consider Algorithm~\ref{alg.DataIndependentSpectralRadiusEstimation} and Algorithm~\ref{alg.NCSStabilityTest}, for given $\delta$ and $\delta_q$. Assume that $N_q$ is selected to be greater than
  \begin{align}
    \max\{ 2| 1-q-\frac{1}{\rho(A)^2}|^{-2} \log \frac{2}{\delta_q},  2 (1-q)^{-2} \log \frac{2}{\delta_q} \}, \label{eq.LowerBoundNq}
  \end{align}
  and define
  \begin{align*}
\epsilon&=
                \frac{1}{2}   \max\left\{  1/\sqrt{{1-q+2f_5(N_q, \delta_q)}}-\rho(A)\right.,\\
    & \qquad \qquad \qquad \qquad  \left. \rho(A) -1/\sqrt{{1-q-2f_5(N_q, \delta_q)}} \right\},\\
    b&=\epsilon-2(1-\frac{1}{n})\|A\|.
\end{align*}
If, further,  $b>0$ and  $N$ is larger than 
  \begin{align}
\max \left\{ \frac{256 \sigma_w^2  (n+2p)}{b^2\lambda_{\min }(\Sigma)}\log\frac{36}{\delta},  \quad 8(n+p)+16\log (\frac{4}{\delta})\right\}, \label{eq.NLowerBounds}
  \end{align}
Algorithm~\ref{alg.NCSStabilityTest} returns the correct answer with probability at least
$(1-\delta)(1-\delta_q)$.
\end{corollary}
\begin{proof}
  If $N_q$ is greater than~\eqref{eq.LowerBoundNq}, we can show that~\eqref{eq:NqStar1} and~\eqref{eq:NqStar2} hold.
  If $N$ is larger than~\eqref{eq.NLowerBounds}, we can follow Corollary~\ref{cor:SpectralRadiusEstimationSampleComplexity} to show that~\eqref{eq.Nstar2} and~\eqref{eq.temp4} hold.
  Therefore, from Theorem~\ref{thm.StabilityTestAnalysis}, we can prove the result.
\end{proof}

  From Corollary~\ref{cor.StabilizabilityTestSampleComplexity}, we can observe that the sample complexity bounds~\eqref{eq.LowerBoundNq} and \eqref{eq.NLowerBounds} are inversely proportional to the square of the stabilizability margin.

\begin{remark}
  In specific stabilizability criteria for NCSs, we also need to check whether the absolute value of the determinant of the system state matrix is smaller than a threshold, see~\cite{RN373, RN1566}.
  We can utilize the determinant perturbation bounds in~\cite{RN11430} and derive algorithms similar to those proposed in this paper to determine these kinds of stabilizability conditions based on finitely many samples.
\end{remark}

\section{Numerical Examples}\label{sec.Simulation}

In this section, we illustrate the performance of the proposed algorithms through numerical examples.
Suppose the system parameters are given by
\begin{align*}
  A=
  \begin{bmatrix}
    1.2 & 0.1 \\
    0 &1
  \end{bmatrix},
        B=
        \begin{bmatrix}
          0\\
          1
        \end{bmatrix}.
\end{align*}
Moreover, the standard deviation for the process noise is $\sigma_w=0.1$.
The system is initialized at $x_0=0$.
The control input is generated from i.i.d.\ distribution $\mathcal{N}(0, \sigma_u^2I)$ with $\sigma_u=0.1$.

First, we evaluate the spectral radius estimation error bounds in Theorem~\ref{thm.SpectralDependentBounds} and Theorem~\ref{thm.SpectralIndependentBounds}.
We consider the data generation procedure described at the beginning of Section~\ref{sec.DataIndependentSpectralEstimate}, and
 select $T=5$, $\delta=0.1$, and values of $N$ from   $100 $  to $1000$.
For each given $N$, we independently generate $100$ data sets, run Algorithm~\ref{alg.DataDependentSpectralRadiusEstimation} and Algorithm~\ref{alg.DataIndependentSpectralRadiusEstimation} and obtain $\rho(\hat{A})$.
Algorithm~\ref{alg.DataDependentSpectralRadiusEstimation} uses all the data points from $N$ trajectories to generate the estimate.
While Algorithm~\ref{alg.DataIndependentSpectralRadiusEstimation} only uses the last input and the last two state measurements from each trajectory to generate the estimate.
Fig.~\ref{fig.theorem1and2} shows the estimation error  $|\rho(A)-\rho(\hat{A})|$ and the error bounds from Theorem~\ref{thm.SpectralDependentBounds} and Theorem~\ref{thm.SpectralIndependentBounds}, respectively.
We can observe from the figures that the data-dependent bounds are tighter as expected.
Moreover, the error  $|\rho(A)-\rho(\hat{A})|$ obtained using Algorithm~\ref{alg.DataDependentSpectralRadiusEstimation} and Algorithm~\ref{alg.DataIndependentSpectralRadiusEstimation} are very similar, showing that Algorithm~\ref{alg.DataIndependentSpectralRadiusEstimation} can be data efficient.
\begin{figure}[htbp]
  \centering
  \includegraphics[width=0.4\textwidth]{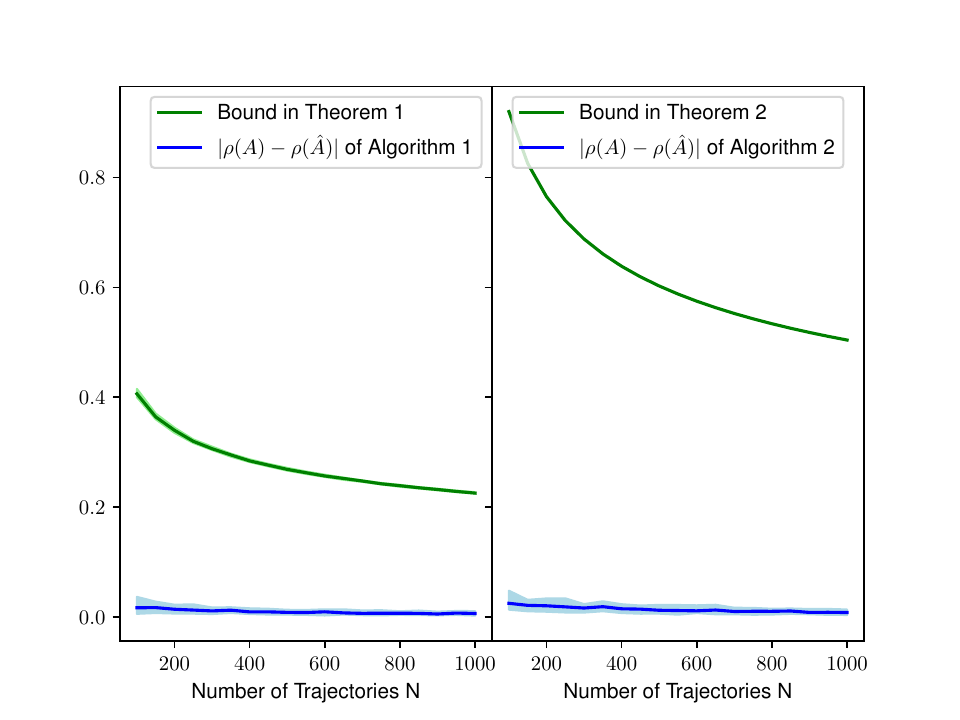}
  \caption{Data dependent and independent bounds of estimates (shown are median, first and third quartiles over 100 independent runs)}
  \label{fig.theorem1and2}
\end{figure}


Next, we evaluate the results in Theorem~\ref{thm.StabilityTestAnalysis} about performance guarantees of Algorithm~\ref{alg.NCSStabilityTest}.
Suppose we already have a high probability estimate of system spectral radius as in~\eqref{eq.temp4} with $\rho(\hat{A})=1.15$, $\epsilon=0.1$ and $\delta=0.01$. 
We select $T=5$, $\delta_q=0.01$ and assume $q=0.75$.
Since $\rho(A)<\sqrt{1/(1-q)}$, we know that the underlying system is stabilizable.
Furthermore, from Theorem~\ref{thm.StabilityTestAnalysis} we can calculate the probability that Algorithm~\ref{alg.NCSStabilityTest} returns the correct answer for any given $N_q$, which we call the \textit{theoretical successful prediction probability} (TSPP). 
Note that if~\eqref{eq:NqStar1}, \eqref{eq:NqStar2} and \eqref{eq.Nstar2} are not satisfied, we set the TSPP to zero.
In parallel, by simulations, we compute the \textit{empirical successful prediction rate} (ESPR) of Algorithm~\ref{alg.NCSStabilityTest}.
We consider values of $N_q$ from  $3$  to $1000$.
For each given $N_q$, we independently generate $400$ data sets, run Algorithm~\ref{alg.NCSStabilityTest} and calculate the ESPR.
These values are shown in Fig.~\ref{fig.NCSStabilityTest}, along with the TSPP.
We can observe that the ESPR quickly becomes $1$ as $N_q$ increases, which demonstrates the effectiveness of Algorithm~\ref{alg.NCSStabilityTest} in checking the stabilizability condition.
Moreover, the TSPP is smaller than the ESPR for all $N_q$, which shows that the results in Theorem~\ref{thm.StabilityTestAnalysis} are valid.

\begin{figure}[htbp]
  \centering
  \includegraphics[width=0.4\textwidth]{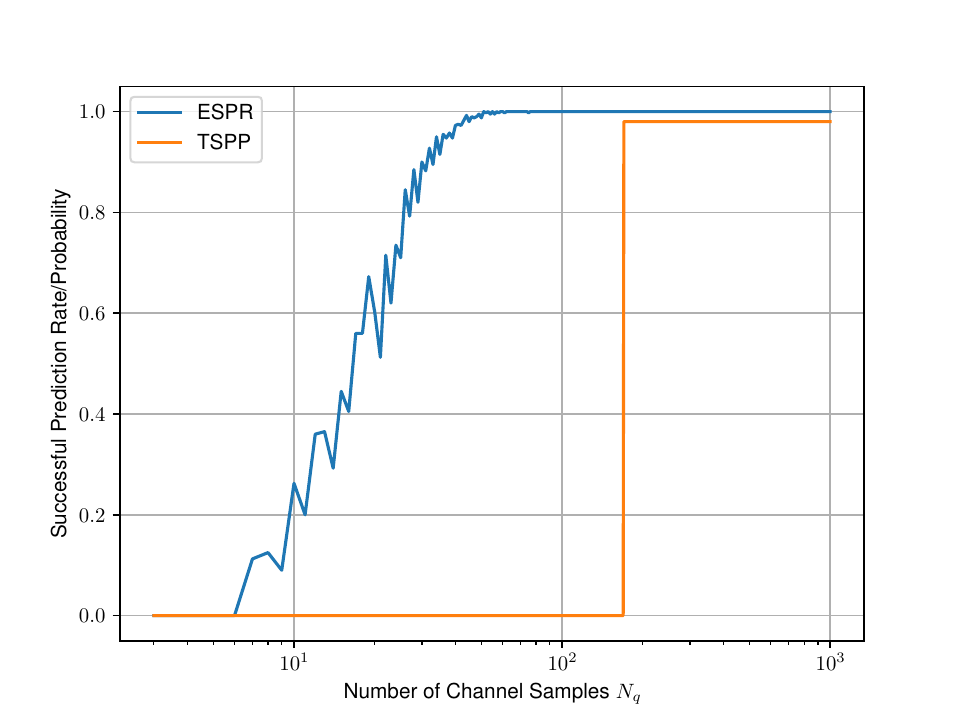}
  \caption{TSPP in Theorem~\ref{thm.StabilityTestAnalysis} and ESPR of Algorithm~\ref{alg.NCSStabilityTest}}
  \label{fig.NCSStabilityTest}
\end{figure}

\section{Conclusion\label{sec.Conclusion}}
We study how to reliably estimate the spectral radius from input and state measurements of a linear time-invariant system.
The derived results rely on existing works of finite sample analysis of least-squares and matrix eigenvalue perturbation techniques.
We further show how to use the proposed methods to determine the stabilizability conditions of NCSs.
Future work will be devoted to directly estimating the spectral radius from the collected data without the need to reconstruct system matrices first and using the proposed methods for data-driven control design for networked systems with performance guarantees.

\bibliographystyle{IEEEtran}
\bibliography{references}

\begin{thebibliography}{10}
\providecommand{\url}[1]{#1}
\csname url@samestyle\endcsname
\providecommand{\newblock}{\relax}
\providecommand{\bibinfo}[2]{#2}
\providecommand{\BIBentrySTDinterwordspacing}{\spaceskip=0pt\relax}
\providecommand{\BIBentryALTinterwordstretchfactor}{4}
\providecommand{\BIBentryALTinterwordspacing}{\spaceskip=\fontdimen2\font plus
\BIBentryALTinterwordstretchfactor\fontdimen3\font minus
  \fontdimen4\font\relax}
\providecommand{\BIBforeignlanguage}[2]{{%
\expandafter\ifx\csname l@#1\endcsname\relax
\typeout{** WARNING: IEEEtran.bst: No hyphenation pattern has been}%
\typeout{** loaded for the language `#1'. Using the pattern for}%
\typeout{** the default language instead.}%
\else
\language=\csname l@#1\endcsname
\fi
#2}}
\providecommand{\BIBdecl}{\relax}
\BIBdecl

\bibitem{RN11260}
V.~K. Mishra, I.~Markovsky, and B.~Grossmann, ``Data-driven tests for
  controllability,'' \emph{IEEE Control Systems Letters}, vol.~5, no.~2, pp.
  517--522, 2021.

\bibitem{RN10935}
H.~J. van Waarde, J.~Eising, H.~L. Trentelman, and M.~K. Camlibel, ``Data
  informativity: a new perspective on data-driven analysis and control,''
  \emph{IEEE Transactions on Automatic Control}, 2020.

\bibitem{RN11100}
A.~Romer, J.~Berberich, J.~Köhler, and F.~Allgöwer, ``One-shot verification
  of dissipativity properties from input–output data,'' \emph{IEEE Control
  Systems Letters}, vol.~3, no.~3, pp. 709--714, 2019.

\bibitem{RN11098}
T.~Maupong, J.~C. Mayo-Maldonado, and P.~Rapisarda, ``On {Lyapunov} functions
  and data-driven dissipativity,'' \emph{IFAC-PapersOnLine}, vol.~50, no.~1,
  pp. 7783--7788, 2017.

\bibitem{RN11099}
A.~Romer, J.~M. Montenbruck, and F.~Allgöwer, ``Determining dissipation
  inequalities from input-output samples,'' \emph{IFAC-PapersOnLine}, vol.~50,
  no.~1, pp. 7789--7794, 2017.

\bibitem{RN11095}
A.~Koch, J.~Berberich, and F.~Allgöwer, ``Verifying dissipativity properties
  from noise-corrupted input-state data,'' in \emph{the 59th IEEE Conference on
  Decision and Control}.\hskip 1em plus 0.5em minus 0.4em\relax Jeju, Korea:
  IEEE, 2020, Conference Proceedings, pp. 616--621.

\bibitem{RN11096}
M.~Sharf, ``On the sample complexity of data-driven inference of the $\mathcal
  {L} _2 $-gain,'' \emph{IEEE Control Systems Letters}, vol.~4, no.~4, pp.
  904--909, 2020.

\bibitem{RN373}
L.~Schenato, B.~Sinopoli, M.~Franceschetti, K.~Poolla, and S.~S. Sastry,
  ``Foundations of control and estimation over lossy networks,''
  \emph{Proceedings of the IEEE}, vol.~95, no.~1, pp. 163--187, 2007.

\bibitem{RN2897}
V.~Gupta, B.~Hassibi, and R.~M. Murray, ``Optimal {LQG} control across
  packet-dropping links,'' \emph{Systems $\&$ Control Letters}, vol.~56, no.~6,
  pp. 439--446, 2007.

\bibitem{RN11554}
J.~Kenanian, A.~Balkan, R.~M. Jungers, and P.~Tabuada, ``Data driven stability
  analysis of black-box switched linear systems,'' \emph{Automatica}, vol. 109,
  p. 108533, 2019.

\bibitem{RN11389}
K.~Gatsis and G.~J. Pappas, ``Statistical learning for analysis of networked
  control systems over unknown channels,'' \emph{Automatica}, vol. 125, p.
  109386, 2021.

\bibitem{RN10374}
S.~Dean, H.~Mania, N.~Matni, B.~Recht, and S.~Tu, ``On the sample complexity of
  the linear quadratic regulator,'' \emph{Foundations of Computational
  Mathematics}, pp. 1--47, 2019.

\bibitem{RN10762}
N.~Matni and S.~Tu, ``A tutorial on concentration bounds for system
  identification,'' in \emph{the 58th IEEE Conference on Decision and Control},
  Nice, France, 2019, Conference Proceedings, pp. 3741--3749.

\bibitem{RN11408}
R.~Bhatia, \emph{Matrix analysis}.\hskip 1em plus 0.5em minus 0.4em\relax
  Springer Science $\&$ Business Media, 2013, vol. 169.

\bibitem{RN11035}
M.~Simchowitz, H.~Mania, S.~Tu, M.~I. Jordan, and B.~Recht, ``Learning without
  mixing: Towards a sharp analysis of linear system identification,'' in
  \emph{the 31st Conference On Learning Theory}, vol.~75.\hskip 1em plus 0.5em
  minus 0.4em\relax PMLR, 2018, Conference Proceedings, pp. 439--473.

\bibitem{RN11104}
T.~Sarkar and A.~Rakhlin, ``Near optimal finite time identification of
  arbitrary linear dynamical systems,'' in \emph{the 36th International
  Conference on Machine Learning}, vol.~97.\hskip 1em plus 0.5em minus
  0.4em\relax Long Beach, California, USA: PMLR, 2019, Conference Proceedings,
  pp. 5610--5618.

\bibitem{RN10404}
L.~Wasserman, \emph{All of statistics: a concise course in statistical
  inference}.\hskip 1em plus 0.5em minus 0.4em\relax Springer Science \&
  Business Media, 2013.

\bibitem{RN1566}
K.~You and L.~Xie, ``Network topology and communication data rate for
  consensusability of discrete-time multi-agent systems,'' \emph{IEEE
  Transactions on Automatic Control}, vol.~56, no.~10, pp. 2262--75, 2011.

\bibitem{RN11430}
I.~C. Ipsen and R.~Rehman, ``Perturbation bounds for determinants and
  characteristic polynomials,'' \emph{SIAM Journal on Matrix Analysis and
  Applications}, vol.~30, no.~2, pp. 762--776, 2008.

\end{thebibliography}
\end{document}